\documentclass{article}

\usepackage[utf8]{inputenc}

\usepackage{latexsym}
\usepackage{amsthm}
\usepackage{amsmath}
\usepackage{amsfonts}
\usepackage{amssymb}
\usepackage{tikz}
\usepackage{pgfplots}
\usetikzlibrary{shapes.multipart}
\usetikzlibrary{calc}
\pgfplotsset{compat=1.15}
\usepackage{subcaption}
\usepackage{float}
\usepackage[linesnumbered,boxed,commentsnumbered,ruled]{algorithm2e}
\usepackage[backend=biber,style=numeric,natbib=true]{biblatex}
\usepackage{hyperref}
\usepackage{blindtext,titlefoot}

\numberwithin{equation}{section}

\newtheorem{thm}{Theorem}[section]

\newtheorem{prop}[thm]{Proposition}
\theoremstyle{definition}
\newtheorem{definition}[thm]{Definition}

\title{Reconstruction of Univariate Functions from Directional Persistence Diagrams}

\author{Aina Ferr\`a$\,{}^{1,\,2}$, Carles Casacuberta$\,{}^{1,\,2}$, Oriol Pujol$\,{}^{1,\,2}$}

\date{}

\begin{document}

\footnotetext[1]{Departament de Matem\`atiques i Inform\`atica, Universitat de Barcelona (UB), Gran Via de les Corts Catalanes, 585, 08007 Barcelona, Spain, \{aina.ferra.marcus, carles.casacuberta, oriol\_pujol\} @ub.edu}

\footnotetext[2]{Supported by MCIN/AEI/10.13039/501100011033 under grant PRE2020-094372 (A.\,Ferr\`a) and projects PID2019-105093GB-I00 (A.\,Ferr\`a, O.\,Pujol) and PID2020-117971GB-C22 (C.\,Casacuberta)}

\unmarkedfntext{\emph{Keywords:} Persistent homology transform, persistence diagram, data reconstruction algorithm}

\unmarkedfntext{\emph{MSC classes:} 55N31, 62R40, 68T07}

\maketitle

\begin{abstract}
We describe a method for approximating a single-variable function $f$ using persistence diagrams of sublevel sets of $f$ from height functions in different directions. We provide algorithms for the piecewise linear case and for the smooth case. 
Three directions suffice to locate all local maxima and minima of a piecewise linear continuous function from its collection of directional persistence diagrams, while five directions are needed in the case of smooth functions with non-degenerate critical points.

Our approximation of functions by means of persistence diagrams is motivated by a study of importance attribution in machine learning, where one seeks to reduce the number of critical points of signals without a significant loss of information for a neural network classifier. 
\end{abstract}

\section*{Introduction}

For a finite geometric simplicial complex $M$ in Euclidean space $\mathbb{R}^d$ with $d\ge 2$, the \emph{persistent homology transform} (PHT), defined in \cite{turneretal}, is a function on the sphere $S^{d-1}$ that associates to each unit vector $v$ the persistence diagram of sublevel sets of $M$ in the direction~$v$. More precisely, one defines
\[
M_t(v)=\{x\in M\mid x\cdot v\le t\}
\]
for $t\in\mathbb{R}$, where the dot denotes scalar product and persistent homology generators of the filtered space $M(v)$ are computed until dimension~$d-1$. The corresponding persistence module has the singular homology $H_*(M_t(v);\mathbb{R})$ at level~$t$, together with translation operators induced by inclusions of sublevel sets. In this paper, we use coefficients in $\mathbb{R}$ and we only focus on zero-dimensional homology $H_0$. The PHT is continuous with respect to any Wasserstein distance on the set of persistence diagrams \cite{curryetal2021}.

It was proved in \cite[Theorem~3.1]{turneretal} that the PHT is injective if $d=2$ or $d=3$, and this result was extended over all dimensions in \cite{curryetal2021} and \cite{ghristetal2018}. That is, for finite simplicial complexes embedded in~$\mathbb{R}^d$, the collection of persistence diagrams of sublevel sets in all possible directions $v\in S^{d-1}$ uniquely determines the given geometric simplicial complex.
See \cite{Solomonetal2018} for extensive information about inverse problems in homological persistence.

In order to produce efficient reconstruction algorithms, one seeks to use a small number of directions to reconstruct any given simplicial complex under suitable general position assumptions. In the case of graphs, three directions suffice, as shown in~\cite{beltonetal}. Bounds on the number of directions needed for reconstruction of compact definable sets in $\mathbb{R}^d$ can be found in \cite{curryetal2021}.

\enlargethispage{0.5cm}

Given a continuous piecewise linear function $f\colon [a,b]\to\mathbb{R}$ with finitely many vertices, the graph
$G(f)=\{(x,y)\in\mathbb{R}^2\mid y=f(x)\}$
can be viewed as a finite geometric simplicial complex of dimension~$1$.
Our goal is to determine the precise location of the local maxima and minima of $G(f)$ assuming knowledge of a persistence diagram for each direction $v\in S^1$. We label a direction $v$ with the angle $\theta$ that the vector $v$ forms with the positive direction of the $x$-axis, and consider only the upper open hemisphere of $S^1$, so that $0<\theta<\pi$. Our algorithm requires precisely three \emph{admissible} directions, in a sense that we make precise.

We also present an algorithm that locates the local maxima and minima of a \emph{smooth} function $f\colon [a,b]\to\mathbb{R}$,
and prove that the algorithm converges to each critical point of $f$ 
assuming that the second derivative $f''$ does not vanish at~critical points.
Although this assumption can be weakened, it simplifies our presentation of results and is sufficiently plausible in practice ---it excludes the occurrence of intervals in which the graph of $f$ is horizontal.

The injectivity of the PHT for graphs of smooth functions on a bounded domain follows from results in~\cite{horwitz}. 
Indeed, the problem of reconstructing a smooth function from the set of its directional persistence diagrams is essentially equivalent to the problem of recovering a smooth function from the set of its tangent lines.
From \cite{horwitz} we quote the fact that, given a sequence of distinct tangent lines $T_j$ converging to a tangent line $T$,
the sequence of intersection points of $T_j$ with $T$ converges to a point where $T$ is tangent to the graph of~$f$.
Our assumption on the second derivative $f''$ at critical points replaces the assumption made in \cite{horwitz} that $f''$ has only finitely many zeros.

Both in the piecewise linear case and in the smooth case, our algorithm produces a piecewise linear approximation of the given function with the same set of critical points.

Motivation for this work came from a study of explainability in the process of classification of signals by a neural network. This study is summarized in Section~\ref{section4} and further details are given in~\cite{ACO2}.
The possibility of reversing information obtained from persistence landscapes in order to partially reconstruct data allowed us to design a system of importance attribution in signal classification tasks. 
Therefore, our method contributes to interpretability in machine learning.

\section{Critical points and critical lines}

A \emph{critical point} of a continuous function $f\colon [a,b]\to \mathbb{R}$ will mean either $(a,f(a))$ or $(b,f(b))$ or any of the local maxima or local minima $(x_i,f(x_i))$ of $f$ in the interior of $[a,b]$. Inflection points are not treated as critical points in this article. We assume that $f$ has only finitely many critical points; in particular, horizontal segments in the graph of $f$ are excluded.
If $f$ is smooth, then for every critical point $(x_i,f(x_i))$ with $a<x_i<b$ the derivative $f'(x_i)$ vanishes, and we additionally impose that $f''(x_i)\ne 0$ at each critical point.

We denote by $S^1$ the unit circle and view each $v\in S^1$ as a complex number $e^{i\theta}$.
Given any subset $X$ of $\mathbb{R}^{2}$, the \emph{sublevel set} of $X$ in the direction $v=e^{i\theta}$ is defined for $t\in{\mathbb R}$ 
and $0<\theta<\pi$ as
\[ X_t(v) = \{(x,y) \in X \mid x\cos\theta+y\sin\theta \leq t\}. \]

That is, $X_t(v)$ is the subset of $X$ consisting of those points that are placed on or below the line $x\cos\theta+y\sin\theta=t$ along the direction~$v$; see Fig.\;\ref{fig1}. Sublevel sets are nested, since $X_s(v)\subseteq X_t(v)$ if $s\le t$, and we denote by $X(v)$ the filtered space with $X_t(v)$ at height~$t$.

\begin{figure}[htb]
\centering
    \begin{tikzpicture}
    \draw[black] (-2, 0) -- (5, 0);
    \draw[black] (0, -2) -- (0, 5);
    \draw[cyan] (0,0) circle (1);
    \node[inner sep= 0pt, outer sep=0] (H) at (-0.5, 2) {};
    \node[inner sep= 0pt, outer sep=0] (I) at (-0.09, 2.82) {};
    \node[inner sep= 0pt, outer sep=0] (E) at (0.5, 4) {};
    \node[inner sep= 0pt, outer sep=0] (J) at (1.36, 1.98) {};
    \node[inner sep= 0pt, outer sep=0] (F) at (2, 0.5) {};
    \node[inner sep= 0pt, outer sep=0] (K) at (2.71, 1.21) {};
    \node[inner sep= 0pt, outer sep=0] (G) at (4, 2.5) {};
    \node[fill, inner sep=1, red, label={[red] right: $v$}] (C) at (0.5, 0.87) {};
    \draw[blue] (-0.5, 2) -- (0.5, 4) -- (2, 0.5) -- (4, 2.5);
    \draw[red, opacity=0.5, line width=2] (-0.5, 2) -- (-0.09, 2.82);
    \draw[red, opacity=0.5, line width=2] (1.36, 1.98) -- (2, 0.5);
    \draw[red, opacity=0.5, line width=2] (2, 0.5) -- (2.71, 1.21);
    \draw[domain=-2:4, variable=\x, red] plot (\x, -0.58*\x + 2.77);
    \draw[domain=-1:3, variable=\x, red] plot (\x, 1.73*\x);
    \node[label={[red] left: \footnotesize 
    $x\cos\theta+y\sin\theta=t$
    }] at(5.25, 0.2) {};
    \node[label={[red]above left: \footnotesize $ y = x\tan\theta$}] at(2.9,4.5) {};
    \node[label={[red]left: $ t$}] at(1.05, 1.7) {};
    \node[label={[blue]left: $ X$}] at(4.3, 2) {};
    \node[label={[red]left: $ X_t(v)$}] at(-0.3, 2.5) {};
    \node[label={[red]left: $ \theta$}] at(0.7, 0.22) {};
    \end{tikzpicture}
    \caption{A sublevel set of a graph $X$ at a height $t$ along a direction~$v$.}
    \label{fig1}
\end{figure}
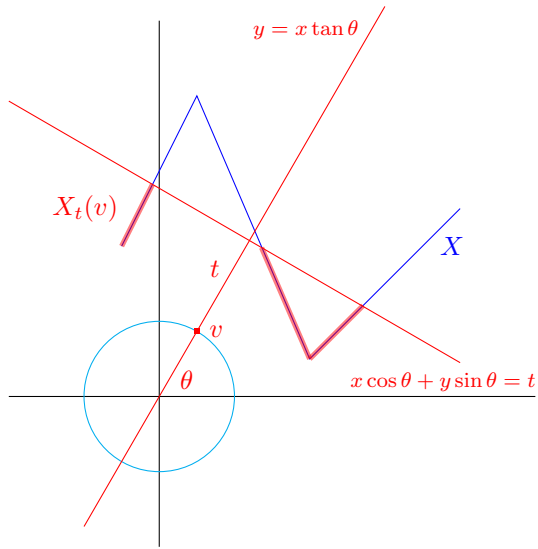

Let $f\colon [a,b]\to {\mathbb R}$ be a continuous function and let $v\in S^1$. We define the \emph{directional persistence module} $M(f,v)$ of $f$ in the direction $v$ as
\[
M_t(f,v)=H_0(G(f)_t(v);{\mathbb R})
\]
for $t\in\mathbb{R}$,
where $G(f)$ is the graph of $f$ with the topology induced by ${\mathbb R}^2$ and $H_0$ is zero-dimensional homology.
Hence $M_t(f,v)$ counts the connected components, for each value of~$t$, of the sublevel sets of the graph of $f$ in the direction~$v$. The translation operators $M_s(f,v)\to M_t(f,v)$ for $s\le t$ are induced by inclusions of sublevel sets; see \cite{CSGO2016} for details about persistence modules.

\begin{definition}
A \emph{critical line} at a height $t$ for a direction $v=e^{i\theta}$ with $0<\theta<\pi$ is a line in $\mathbb{R}^2$ orthogonal to $v$ (hence of slope $-1/\tan\theta$) such that the filtered space $G(f)(v)$ changes its number of connected components at height~$t$, where height $0$ corresponds to the line passing through $(0,0)$. 
\end{definition}

Lines passing through $(a,f(a))$ or $(b,f(b))$ may not be critical. In fact, for the vertical height function these lines will be critical if and only if $(a,f(a))$ respectively $(b,f(b))$ are local minima.

In the smooth case, critical lines are tangent to the graph of~$f$, except perhaps those passing through $(a,f(a))$ or $(b,f(b))$. 
In the piecewise linear case, each critical line in any direction contains at least one critical point of~$f$. 

Each directional persistence module $M(f,v)$ for a continuous function $f\colon [a,b]\to\mathbb{R}$ gives rise to a \emph{barcode} and a \emph{persistence diagram} as in \cite{EH2008,ghrist2008}. When two connected components of $G(f)(v)$ merge at some height, the component which was born more recently (i.e., at a greater height) is annihilated, and if several components were born at the same height then the one with smaller $x$-value is kept. This convention yields a \emph{barcode template} for homological dimension~$0$ in the sense of \cite[Definition~4.3]{LOT}. Hence, a point $(\beta,\delta)$ in a directional persistence diagram represents a connected component of $G(f)(v)$ which is born at a height $\beta$ and merges with an older component at a height~$\delta$. The possibility that $\delta=\infty$ is not discarded; in fact, since the graph of $f$ is connected, each directional persistence diagram contains precisely one point $(\beta,\infty)$, where $\beta$ is the birth height of the whole graph.

A persistence diagram in a direction $v$ tells us precisely which lines orthogonal to $v$ are critical.

\begin{prop}
\label{first}
For a direction $v$, a line orthogonal to $v$ at a height $t$ is critical for a function $f$ if and only if either $t=\beta$ for some point $(\beta,\delta)$ in the persistence diagram for zero-homology $H_0$ of sublevel sets of $f$ in the direction~$v$,
or $t=\delta$ when $\delta$ is finite.
\end{prop}

\begin{proof}
If there is a point $(\beta,\delta)$ with $\delta$ finite in the persistence diagram, then at least one new connected component of $G(f)(v)$ arises at height $\beta$ (so the orthogonal line at height $\beta$ is critical), and at least two connected components merge at height $\delta$, which implies that the orthogonal line at height $\delta$ is also critical. If $\delta=\infty$, then only the line at height $\beta$ is critical.
Conversely, if a line at height $t$ is critical, then either a new connected component arises or two connected components merge at~$t$.
\end{proof}

Persistence diagrams in this article have been calculated using GUDHI (Geometry Understanding in Higher Dimensions \cite{gudhi:urm}), a generic open source library with a Python interface.

\section{Reconstruction algorithms}

\subsection{The piecewise linear case}

In \cite{Fasyetal2019} an algorithm was given for reconstructing an embedded simplicial complex of arbitrary dimension using a finite number of directional persistence diagrams.
In our case, three directional persistence diagrams suffice to reconstruct a continuous piecewise linear function $f$ by providing 
an approximation that coincides with $f$ on its critical points.

Assuming that $f$ has a total number of $n$ critical points, the number of critical lines in each direction is less than or equal to $n$. A critical line can pass through two or more critical points, and a line passing through a critical point can fail to be critical. 
Indeed, if the graph of $f$ crosses transversally 
a line at some critical point~$P$, then the corresponding directional persistence diagram will not include any information relative to~$P$, since there is no change in the number of connected components of the sublevel sets $G(f)_t(v)$ when the line passes through~$P$.
Since we need to exclude this possibility from our algorithm, we give the following definition.

\begin{definition}
For a continuous piecewise linear function $f\colon [a,b]\to {\mathbb R}$, we say that a direction $v=e^{i\theta}\in S^1$ with $0< \theta< \pi$ is \emph{admissible} if 
for each critical point $P=(p_1,p_2)$ of $f$ 
there is an $\varepsilon>0$ such that all the points $(x,f(x))$ on the graph of $f$ with $p_1-\varepsilon<x<p_1+\varepsilon$ and $x\ne p_1$ are on the same open half-plane in $\mathbb{R}^2\smallsetminus L$ where $L$ is the line orthogonal to $v$ containing~$P$.
\end{definition}

If $v$ is admissible then all lines orthogonal to $v$ containing critical points of $f$ are critical lines, except perhaps those at the endpoints $a$ and~$b$.

\begin{figure}[htb]
\hspace{0.5cm}
 \begin{subfigure}[b]{0.45\textwidth}
 \hspace{0.8cm}
    \begin{tikzpicture}[scale=0.9]
    \draw[black] (-0.5, 0) -- (5, 0);
    \draw[black] (0, -0.5) -- (0, 5);
    
    \node[inner sep=1, color=blue, fill] at (0, 2.5) {};
    \node[inner sep=1, color=blue, fill] at (0.5, 4) {};
    \node[inner sep=1, color=blue, fill] at (2, 0.5) {};
    \node[inner sep=1, color=blue, fill] at (4, 2.5) {};
    
    \draw[blue] (0, 2.5) -- (0.5, 4) -- (2, 0.5) -- (4, 2.5);
    
     \draw[domain=-0.5:4, variable=\x, red] plot (\x, -0.578*\x + 1.656);
     \draw[domain=-0.5:5, variable=\x, red] plot (\x, -0.578*\x + 2.5);
     \draw[domain=-0.5:5, variable=\x, red] plot (\x, -0.578*\x + 4.29);
     \draw[domain=-0.5:5, variable=\x, red] plot (\x, -0.578*\x + 4.8121);
     \draw[domain=-0.5:3, variable=\x, black] plot (\x, 1.73*\x);
    \end{tikzpicture}
    \caption{An admissible direction.}
    \end{subfigure}
    \hspace{-0.2cm}
    \begin{subfigure}[b]{0.45\textwidth}
    \hspace{0.6cm}
    \begin{tikzpicture}[scale=0.9]
    \draw[black] (-0.5, 0) -- (5, 0);
    \draw[black] (0, -0.5) -- (0, 5);
    
    \node[inner sep=1, color=blue, fill] at (0, 2.5) {};
    \node[inner sep=1, color=blue, fill] at (0.5, 4) {};
    \node[inner sep=1, color=blue, fill] at (2, 0.5) {};
    \node[inner sep=1, color=blue, fill] at (4, 2.5) {};
    
    \draw[blue] (0, 2.5) -- (0.5, 4) -- (2, 0.5) -- (4, 2.5);
    
     \draw[domain=3.9:4.3, variable=\x, red] plot (\x,  -10*\x + 42.5);
     \draw[domain=1.7:2.08, variable=\x, red] plot (\x, -10*\x + 20.5);
     \draw[domain=0.4:0.95, variable=\x, red] plot (\x,-10*\x + 9);
     \draw[domain=-0.2:0.33, variable=\x, red] plot (\x, -10*\x + 2.5);
     \draw[domain=-0.5:5, variable=\x, black] plot (\x, 0.1*\x);
    \end{tikzpicture}
    \caption{A non-admissible direction.}
    \end{subfigure}
    \caption{A direction $v$ is not admissible if some line orthogonal to $v$ is transverse to the graph at some critical point.}
\end{figure}
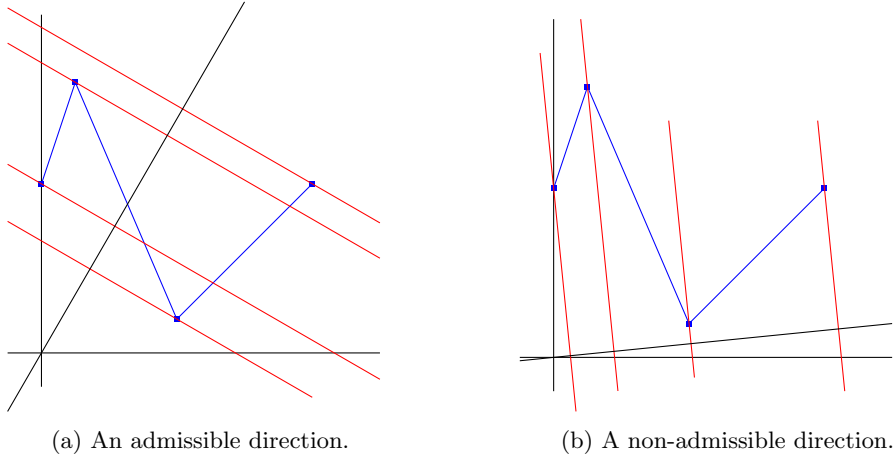

If there are critical points in which the graph of $f$ has very small slopes in absolute value, then the only admissible directions are those with $\theta$ close to~$\pi/2$. More precisely, the following observation guarantees the existence of admissible directions.

\begin{prop}
Let $m$ be the smallest slope in absolute value between two consecutive vertices in the graph of $f$. Then every direction $e^{i\theta}$ with $0<\theta<\pi/2$ and $\tan\theta>1/m$ is admissible.
\end{prop}

\begin{proof}
Since, by assumption, the graph of $f$ does not contain any horizontal segments, we have that $m>0$. Given $v=e^{i\theta}$ with $0<\theta<\pi/2$, the lines orthogonal to $v$ have slope $-1/\tan\theta$. If one such line $L$ contains a critical point $P$, then the admissibility condition is fulfilled at $P$ since the slope of $L$ is smaller in absolute value than the slope of each segment of the graph of $f$ at~$P$.
\end{proof}

The intersection points of two distinct critical lines will be called \emph{double points} and the intersection points of three distinct critical lines will be called \emph{triple points}.

In what follows we fix three directions $v_0=e^{i\theta_0}$, $v_1=e^{i\theta_1}$, $v_2=e^{i\theta_2}$ with 
$0<\theta_2<\theta_1<\theta_0\le\pi/2$,
and describe a procedure to recover a continuous piecewise linear function $f\colon [a,b]\to {\mathbb R}$ from the corresponding directional persistence diagrams. 
Although the algorithm accepts as input any values of $\theta_0$, $\theta_1$ and $\theta_2$, we normally use
$\theta_0=90^\circ$, $\theta_1=85^\circ$and $\theta_2=80^\circ$ (angles given in degrees).
If some of these happens to be non-admissible, then either the angle should be increased or the reconstruction will omit the conflicting critical points.

\begin{thm}
\label{TheTheorem}
Let $f\colon [a,b]\to {\mathbb R}$ be a continuous piecewise linear function, and let $v_0=e^{i\theta_0}$, $v_1=e^{i\theta_1}$ and $v_2=e^{i\theta_2}$ be admissible directions such that $0<\theta_2<\theta_1<\theta_0\le \pi/2$.
Let $\mathcal P$ be the set of triple points determined by the three directional persistence diagrams. 
Suppose that no critical line orthogonal to $v_0$, $v_1$ or $v_2$ contains two or more points from $\mathcal P$.
Then $\mathcal P$ is equal to the set of critical points of~$f$, excluding the boundary points $(a,f(a))$ and $(b,f(b))$ if these are local maxima.
\end{thm}

\begin{proof}
Since the three directions are admissible, every critical point of $f$ (except perhaps the boundary points) is a triple point, since there is a critical line in each direction passing through it.

Conversely, every triple point must be a critical point of~$f$, since if a triple point $P$ is not critical, then each of the three critical lines passing through $P$ contains at least two triple points. 
\end{proof}

Consequently, by joining each pair of consecutive points in $\mathcal P$ (ordered by their $x$-coordinate) we obtain a continuous piecewise linear function with the same critical points as~$f$. Since the boundary points $(a,f(a))$ and $(b,f(b))$ remain undetected unless they are local minima,
in order to entirely reconstruct the graph of $f$ it will be necessary to give its boundary points as part of the oracle's data.

The assumption that no critical line contains two or more triple points does not restrict significantly the validity of our algorithm, since the probability that it fails is negligible. 
Indeed, as in \cite[\S\,5.1]{LOT}, the subset of $S^1$ of those directions for which there are critical lines passing through two or more triple points has measure zero.

Thus the algorithm consists of finding triple intersections within the set of critical lines determined by three chosen directional persistence diagrams. There is a naive algorithm consisting of looking at all the possible intersections between all such critical lines, that would take $O(n^3)$ where $n$ is the number of critical points of the original function. However, since there is a strong geometrical component, we can improve performance by lowering the complexity to at most $2n\log n+2n^2$ operations, hence~$O(n^2)$, as follows.

\begin{figure}[htb]
    \centering
    \begin{tikzpicture}[scale=0.95]
    \begin{axis}[every axis plot post/.append style={
      mark=none,domain=-6:4,samples=95,smooth}, 
      axis x line*=middle, 
      axis y line*=middle, 
      ymin=-2,
      ymax=4.5,
      ticks = none] 
      \node[color=green,left] (AA) at(5, 2){$T_0$};
      \addplot[green] {2};
      \node[color=green, left] (AA) at(5, 1){$T_1$};
      \addplot[green] {1};
      \node[color=green, left] (AA) at(5, 0.5){$T_2$};
      \addplot[green] {0.5};
      \node[color=red, left] (BB) at(-6, 2.2){$S_i$};
      \addplot[red] {-0.18*x+1.02};
      \node[color=violet, left] (CC) at(5, 0){$R_2$};
      \addplot[violet] {-0.36*x + 1.5};
      \node[color=violet, left] (CC) at(5, -0.5){$R_1$};
      \addplot[violet] {-0.36*x + 1.06};
      \node[color=violet, left] (CC) at(5, -1){$R_0$};
      \addplot[violet] {-0.36*x + 0.5};
    \end{axis}
    \end{tikzpicture}
    \caption{Ordering of critical lines in three directions.}
    \label{fig:lattice_Rec}
\end{figure}
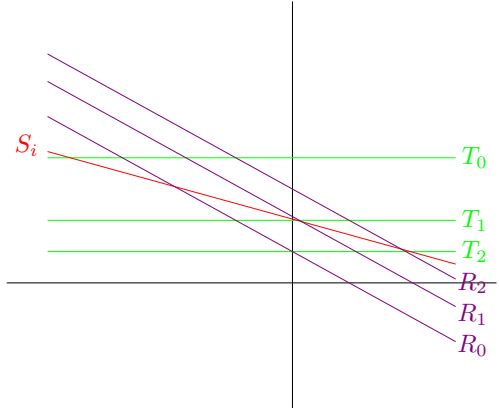
For each critical line $S_i$, the lines $T_j$ and $R_j$ are moved simultaneously to search for triple intersection points. To optimize the algorithm, each of the sets of lines $T$ and $R$ has to be ordered. Since 
$S_i$ has negative slope, the $T_j$ 
lines have to be ordered decreasingly (with $T_0$ above the others) while the $R_j$ 
lines must be ordered increasingly (with $R_0$ below the others). 

\medskip

\noindent
\textbf{The rolling ball algorithm.}
For every $S_i$, we start by intersecting  it with the first lines $T_0$ and $R_0$, obtaining
$P_t = S_i \cap T_0$ and $P_r = S_i \cap R_0$.
If $P_t = P_r$, a triple point has been found and we can move to $S_{i+1}$. If not, we check which point has a lower value of its $x$-coordinate (hence the name of the algorithm). 
If $x(P_t) < x(P_r)$, as depicted in Fig.\;\ref{fig:lattice_Rec}, then the line $T_0$ can be discarded: since $S_i$ has negative slope and the lines $T$ and $R$ are ordered, a triple intersection cannot occur with $T_0$.

In general, at each step of the algorithm, given a line $S_i$ we compare, in the 
prescribed order, $T_j$ with $R_k$ and check for the intersections $P_t$ and $P_r$. If $P_t \neq P_r$, then we move to compare $T_{j+1}$ and $R_k$ if $x(P_t) < x(P_r)$ and we move to compare $T_{j}$ and $R_{k+1}$ otherwise.

In order to compare the speed of the naive algorithm with our rolling ball algorithm, we used a bank of functions with a varying number of critical points. For that purpose, we programmed a snippet of code that returns a piecewise linear function with the desired number of points. 

We ran both algorithms with functions with the following number of critical points: $5$, $10$, $25$, $50$, $100$, $150$, $200$, and recorded the time in seconds each algorithm takes for each function. 
Table \ref{table:comparison} shows the comparison between the two algorithms.
The difference in optimization is really noticeable. In fact, when dealing with $200$ critical points, the naive reconstructing algorithm took a bit over $300$ seconds while the optimized algorithm took only $1.6$ seconds. The difference between time is what makes the optimized algorithm a feasible algorithm to be applied to large quantities of data. 
\begin{table}[H]
\centering
\begin{tabular}{@{}rrrrrrrr@{}}
& \textbf{5} & \textbf{10} & \textbf{25} & \textbf{50} & \textbf{100} & \textbf{150} & \textbf{200} \\ \hline \\[-0.3cm]
\textbf{Naive}     & 0.008      & 0.026       & 0.455       & 4.431        & 36.34         & 127.1        & 301.7        \\
\textbf{Optimized} & 0.002      & 0.003       & 0.017       & 0.115       & 0.278        & 0.835        & 1.635         \\[0.05cm] \hline
\end{tabular}
\caption{Comparison of a naive reconstruction algorithm with the rolling ball algorithm. Time is recorded in seconds. The upper row indicates the number of critical points of the given piecewise linear functions.}
\label{table:comparison}
\end{table}

\subsection{The smooth case}

Suppose given a smooth function $f\colon [a,b]\to\mathbb{R}$, and assume that $f''(x_i)\ne 0$ for all critical points $(x_i,f(x_i))$.
Hence we discard the possibility that $f$ is ``too flat'' at some critical point (similarly as in the piecewise linear case), and only isolated critical points are considered.

Let $T$ be a tangent line to $f$ at a critical point $(x_0,y_0)$. By \cite{horwitz}, $(x_0,y_0)$ is a limit point of intersection points of tangent lines $T_j$ close to $T$ (in the sense that both the slope and the intercept of $T_j$ are close to those of $T$). In our situation, such tangent lines are provided by directional persistence diagrams approaching the vertical height direction as much as needed (so the tangent lines become close to horizontal).

Our algorithm works as follows.
If $(x_0,y_0)$ is a critical point of $f$ with $a<x_0<b$, then $f'(x_0)=0$ and the horizontal line $y=y_0$ is a critical line in the vertical direction. Hence, in order to obtain  $x_0$ from directional persistence diagrams, it suffices that the algorithm detects the existence of $(x_0,y_0)$ and that it converges to~$x_0$. We treat detection and convergence separately.

\medskip

\noindent
\textbf{Detection of critical points.}
Let $(x_0,y_0)$ be a critical point of $f\colon [a,b]\to\mathbb{R}$ with $a<x_0<b$. For a direction $v=e^{i\theta}$ with $0<\theta<\pi/2$ and a positive real number $\tau>0$, we say that $(x_0,y_0)$ is \emph{$\tau$-detected} by the direction $v$ if there are critical lines of slopes $\pm1/\tan\theta$ intersecting $y=y_0$ at points $(x_1,y_0)$ and $(x_2,y_0)$ with $x_1<x_0<x_2$ and $x_2-x_1<\tau$ (see Fig.\;\ref{triangle}), and there is no other critical point of $f$ whose $x$-coordinate lies between $x_1$ and $x_2$.

\begin{prop}
\label{Detection}
Suppose that $(x_0,y_0)$ is a critical point of a smooth function $f\colon [a,b]\to \mathbb{R}$.
Suppose that $|f''(x)|>\varepsilon$ if $|x-x_0|<\delta$ for some numbers $\delta>0$ and $\varepsilon>0$.
If we pick an angle $0<\theta<\pi/2$ such that $\tan\theta>1/\delta\varepsilon$ and a number $\tau$ with
$\tau>2\delta$, then $(x_0,y_0)$ is $\tau$-detected by~$e^{i\theta}$.
\end{prop}

\begin{proof}
Suppose first that $f''(x_0)>0$.
Note first that
\[
f'(x_0+\delta)=\int_{x_0}^{x_0+\delta}f''(t)\,dt>\delta\varepsilon>\frac{1}{\tan\theta}.
\]
Hence there is a value $a$ with $x_0<a<x_0+\delta$ such that the tangent to the graph of $f$ at $(a,f(a))$ has slope $1/\tan\theta$. Now the Lagrange formula for the remainder of the first-order Taylor polynomial of $f$ at $x_0$ yields
\[
\textstyle
f(a)=y_0+\frac12f''(c)(a-x_0)^2>y_0+\frac12\varepsilon(a-x_0)^2
\]
for some $c$ with $x_0<c<a$.
The tangent to the graph of $f$ at $(a,f(a))$ intersects $y=y_0$ at a point $(x_2,y_0)$ and the parallel line through $(a,\,y_0+\frac12\varepsilon(a-x_0)^2)$ intersects $y=y_0$ at $(x_*,y_0)$ with $x_*=a-\frac12\varepsilon\tan\theta(a-x_0)^2$. Since $\tau>2/\varepsilon\tan\theta$ and $a-x_0<\delta<\frac12\tau$, we find that
\begin{align*}
x_2-x_0 & <x_*-x_0=
(a-x_0)\big[1-\textstyle\frac12\varepsilon\tan\theta(a-x_0)\big] \\ & <
(a-x_0)\Big(1-\frac{a-x_0}{\tau}\Big)<\textstyle\frac12 (x_0+\tau-a)<\textstyle\frac12 \big(\textstyle\frac12\tau+\frac12\tau\big)=\frac12\tau.
\end{align*}
By symmetry, there is a tangent with slope $-1/\tan\theta$ at a point $(b,f(b))$ with $x_0-\delta<b<x_0$ intersecting $y=y_0$ at a point $(x_1,y_0)$ with $x_0-x_1<\frac12\tau$. Consequently, $x_2-x_1<\tau$. Moreover, there is no other critical point whose $x$\nobreakdash-co\-ord\-inate lies between $x_1$ and $x_2$ since $f'$ is monotone in this interval. Hence, $(x_0,y_0)$ is $\tau$-detected by the direction $e^{i\theta}$.
Again by symmetry, the conclusion is the same if $f''(x_0)<0$.
\end{proof}

\noindent
\textbf{Convergence.}
Suppose that $(x_0,y_0)$ is a local \emph{minimum} that has been detected by the algorithm. Let $L_0$ be the corresponding horizontal critical line $y=y_0$, and let $v_1=e^{i\theta_1}$ be the direction involved in the successful detection process.
Choose another direction $v_2=e^{i\theta_2}$ with
$0<\theta_1<\theta_2<\pi/2$ 
(in fact, we pick $\theta_2$ sufficiently close to $\pi/2$).

Let us denote 
$m_1=1/\tan\theta_1$ and $m_2=1/\tan\theta_2$,
and let $L_1$, $L_2$, $L_3$, $L_4$ be the corresponding critical lines of respective slopes $-m_1$, $m_1$, $-m_2$ and $m_2$.
Let $(x_i,y_0)$ denote the intersection point of $L_i$ with $L_0$ for $i=1,2,3,4$.

The $x$-coordinate $x_0$ of the critical point $M$ is located between $x_3$ and $x_4$ (although it need not be the midpoint), and we can approximate it by considering the line $L$ passing through the vertices $P$ and $Q$ of the triangles formed by $L_0,L_1,L_2$ and by $L_0,L_3,L_4$ respectively.
The intersection point $M_*=L\cap L_0$ is $(x_*, y_0)$, where
\begin{equation}
\label{mmm}
x_*=\frac{m_2(x_1+x_2)(x_3-x_4)-m_1(x_1-x_2)(x_3+x_4)}{2m_2(x_3-x_4)-2m_1(x_1-x_2)}.
\end{equation}
This is taken as an approximation to the true $x$-coordinate of $M=(x_0,y_0)$.

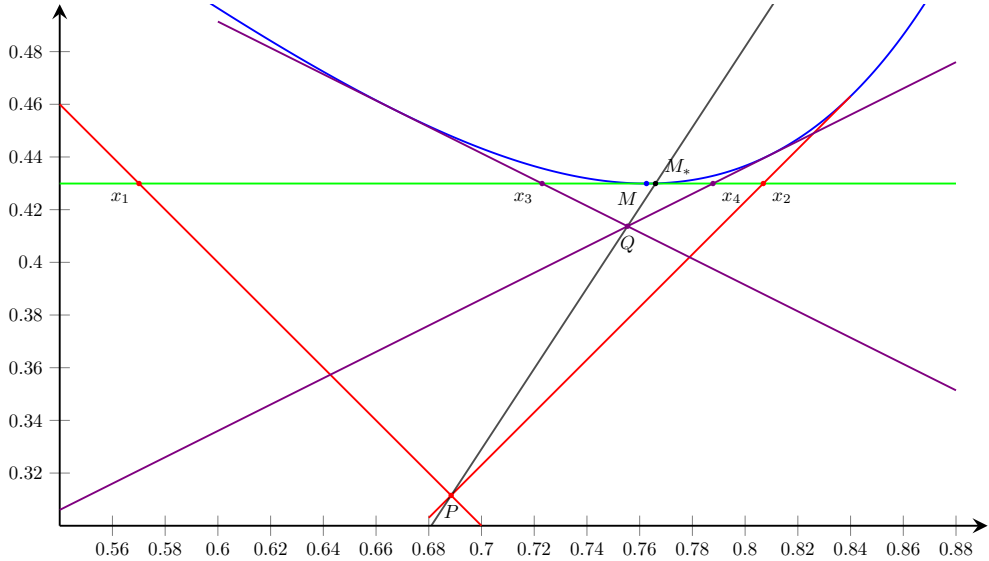
\begin{figure}[htb]
    \centering
    \begin{tikzpicture}[scale=1.8, every node/.style={scale=0.4}]
        \begin{axis}[every axis plot post/.append style={
        mark=none,samples=100,smooth}, 
        axis x line=middle, 
        axis y line=middle, 
        enlargelimits=upper,
          xmin=0.54, xmax=0.86,
          ymin=0.30, ymax=0.48,
          xtick distance=0.02,
          ytick distance=0.02,
          axis equal image,
          grid = none] 
          \addplot[domain=0.54:0.88, color=blue] {x^9 - x^4 + x^3 - x + 1};
          \addplot[domain=0.54:0.88, color=green] {0.429917};
          \addplot[domain=0.54:0.7, color=red] {-x + 0.570083 + 0.429917};
          \addplot[domain=0.68:0.84, 
          color=red] {x - 0.806887 + 0.429917};
          \addplot[domain=0.60:0.88,
          color=violet] {-0.5*x + 0.5*0.722979 + 0.429917};
          \addplot[domain=0.54:0.88, 
          color=violet] {0.5*x - 0.5*0.787817 + 0.429917};
          \addplot[domain=0.68:0.82, color=black, opacity=0.7] {1.5272443322*x - 0.73997};
          \node[fill, color=red, circle, inner sep=1, label={below left: $x_1$}] at (0.570083, 0.429917) {};
          \node[fill, color=red, circle, inner sep=1, label={below right: $x_2$}] at (0.806887, 0.429917) {};
          \node[fill, color=violet, circle, inner sep=1, label={below left: $x_3$}] at (0.722979, 0.429917) {};
          \node[fill, color=violet, circle, inner sep=1, label={below right: $x_4$}] at (0.787817, 0.429917) {};
          \node[fill, color=red, circle, inner sep=1, label={below: $P$}] at (0.688485, 0.311515 ) {};
          \node[fill, color=violet, circle, inner sep=1, label={below: $Q$}] at (0.755398, 0.4137075) {};
          
          \node[fill, color=blue, circle, inner sep=1, label={below left: $M$}] at (0.762580, 0.429917) {};
          \node[fill, color=black, circle, inner sep=1, label={above right: $M_*$}] at (0.766012, 0.429917) {};
        \end{axis}
    \end{tikzpicture}
    \caption{An approximation $M_*$ to the critical point $M$ is obtained by means of five critical lines. The point $M_*$ is the intersection of the line $L=PQ$ with the horizontal tangent $L_0$. Here $P$ is the intersection of tangents with slopes $\pm 45^\circ$ and $Q$ is obtained using slopes of $\pm 30^\circ$. The function represented in the picture is $f(x)=x^9 - x^4 + x^3 - x + 1$.}
    \label{triangle}
\end{figure}

This step can be repeated with smaller slopes in order to achieve any desired precision, and the case of a local maximum is treated analogously, by symmetry. 

\begin{thm}
\label{SmoothTheorem}
Let $f\colon [a,b]\to{\mathbb R}$ be a smooth function, and let $(x_0,y_0)$ be a critical point of $f$ with $a<x_0<b$ and such that $f''(x_0)\ne 0$.
Then $x_0$ can be approximated with any arbitrary degree of precision by means of critical lines.
\end{thm}

\begin{proof}
If $f''(x_0)\ne 0$ then there exist $\varepsilon>0$ and $\delta>0$ such that $|f''(x)|>\varepsilon$ if $|x-x_0|<\delta$.
Use Proposition~\ref{Detection} to find an angle $\theta$ with $0<\theta<\pi/2$ and a number $\tau>0$ such that $(x_0,y_0)$ is $\tau$-detected by $e^{i\theta}$. Next, choose a sequence of pairs of critical lines $L_{2k-1}$, $L_{2k}$ tangent to the graph of $f$ with slopes $\pm1/\tan\theta_k$ tending to zero, and intersecting the line $y=y_0$ at points within $(x_0-\frac12\tau,\,x_0+\frac12\tau)$. If we calculate \eqref{mmm} in each case, the resulting sequence converges to $x_0$ as in~\cite{horwitz}.
\end{proof}

\medskip

\noindent
\textbf{The five-line algorithm.}
The first step of the algorithm aims to detect all critical points, even with a risk of introducing false positives. Since the functions are assumed to be unknown and the algorithm depends on an oracle with information from directional persistence diagrams, we use the same values of $\theta$ and $\tau$ for all critical points. All triangles with base less than $\tau$ formed by critical lines are tentatively assumed to contain a critical point. The algorithm usually starts with a slope of~$30^\circ$ and a value $\tau = 0.08$. 

\begin{figure}[htb]
    \centering
    
        \begin{subfigure}{0.44\textwidth}
          \centering
          \includegraphics[width=0.8\textwidth]{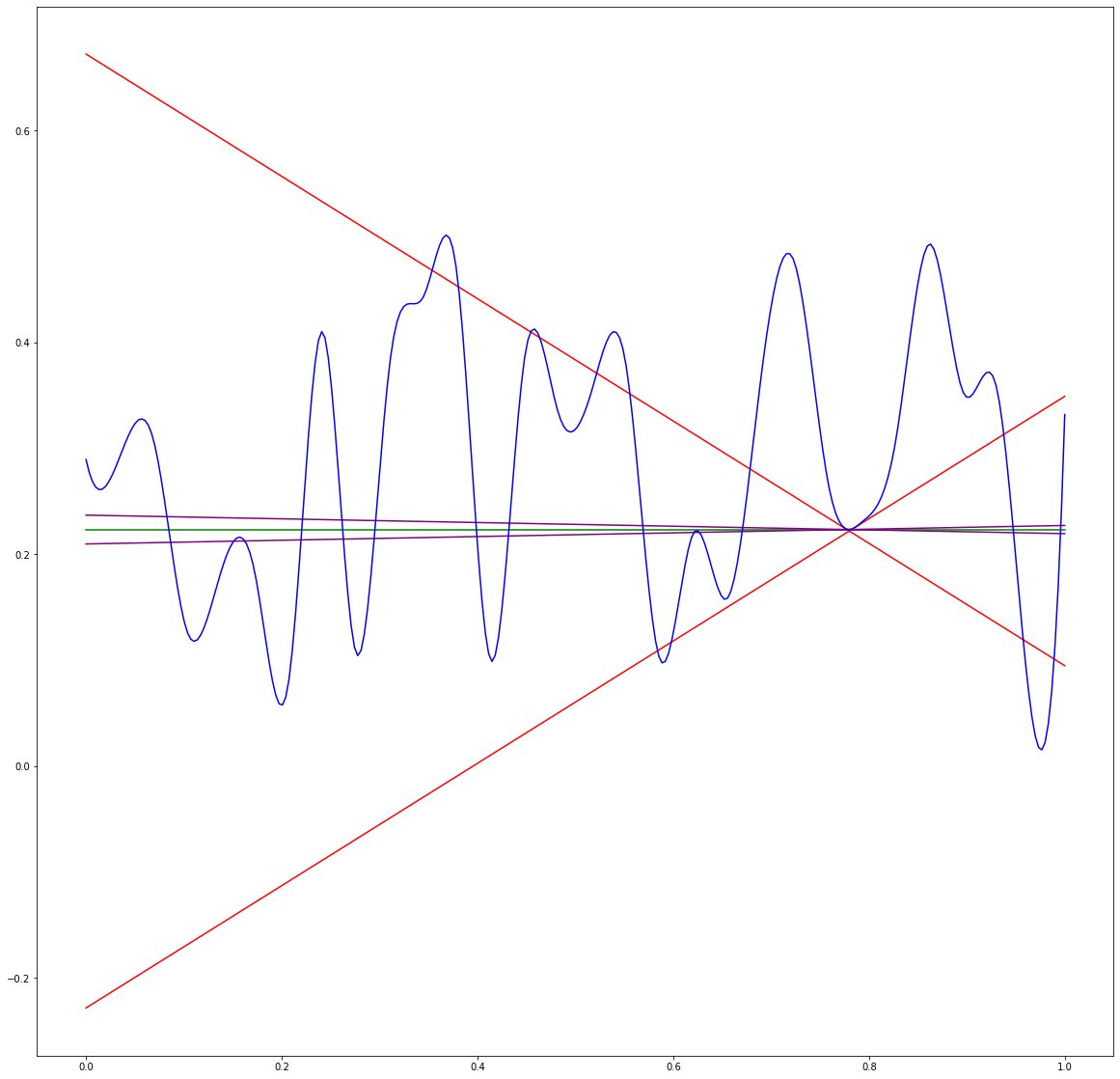} 
          \caption{Location of an asymmetric critical point by means of five tangent lines.}
          \label{crit_a}
        \end{subfigure} \hspace{2em}
        \begin{subfigure}{0.44\textwidth}
          \centering
          \includegraphics[width=0.8\textwidth]{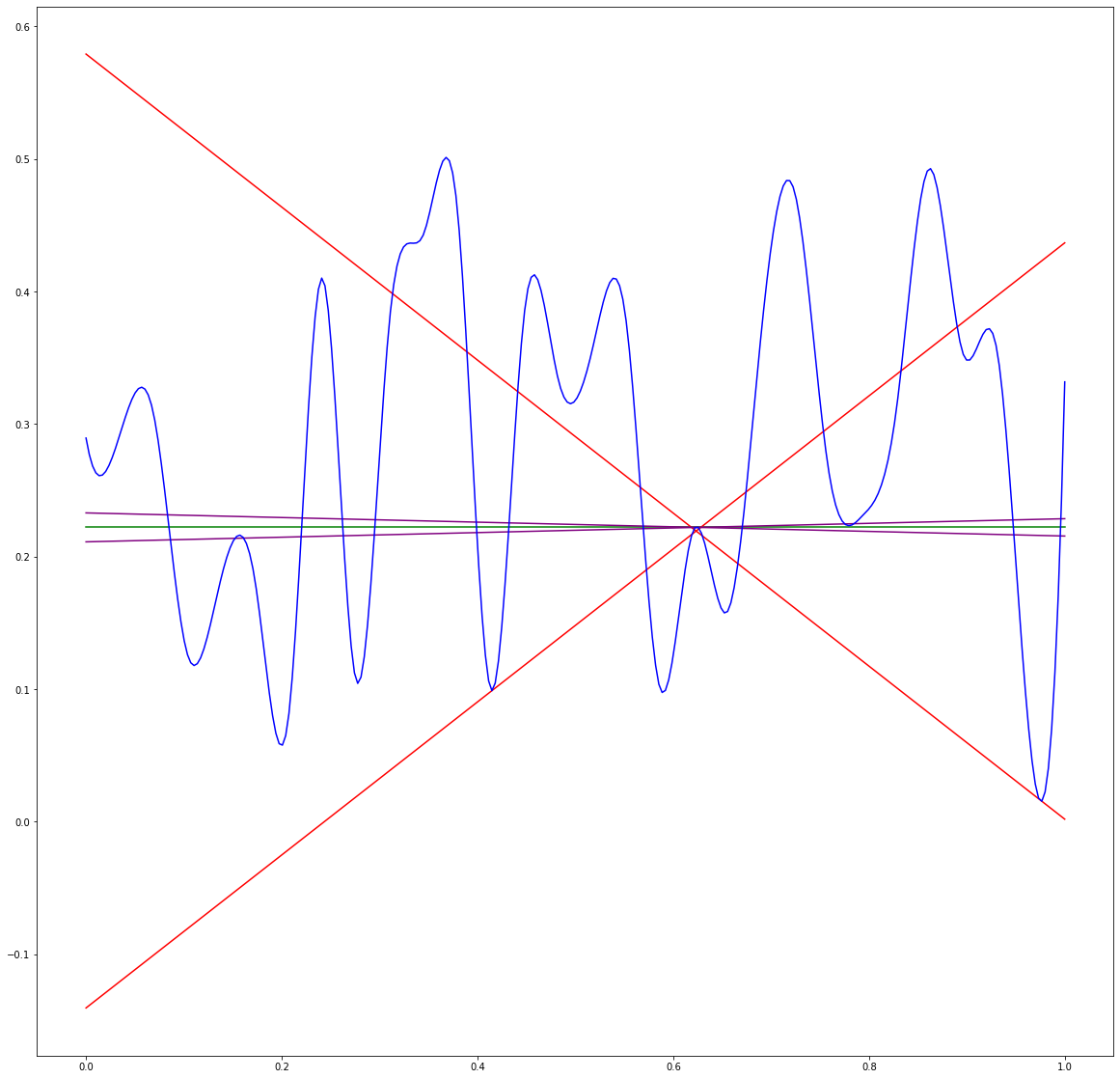}
          \caption{A critical point coincidentally detected by quasi-multiple tangent lines (red).}
        \label{crit_b}
        \end{subfigure}
        
        \begin{subfigure}{\textwidth}
          \centering
          \includegraphics[width=0.48\textwidth]{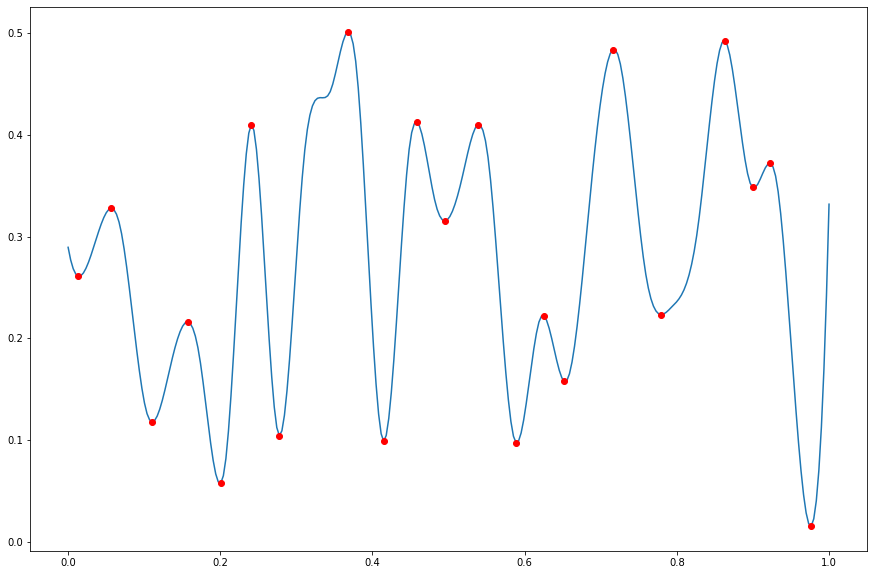}  
          \caption{The algorithm finds all critical points except perhaps endpoints (inflection points are not considered).}
          \label{crit_c}
        \end{subfigure} \hspace{2em}
        
    \caption{Instances of the five-line algorithm for smooth functions.}
    \label{fig:rec_process}
\end{figure}

The second step of the algorithm tries to eliminate false positives and accurately locate each critical point for true positives. Our working hypothesis is that critical lines with slope close to $0$ will only pass inside a triangle from the first step 
(that is, $x_3$ and $x_4$ are between $x_1$ and $x_2$ as depicted in Fig.\;\ref{triangle})
if that triangle truly contains a critical point. Typically, we use a slope of $0.1^\circ$ for that step, and we disregard candidate triangles where betweenness fails. Finally, we apply the formula \eqref{mmm} to every resulting pair of triangles to approximate the enclosed critical point.

The choice of parameters depends on the nature of the signals to which the algorithm is to be applied. In the first step of the algorithm, the choice of $\tau$ is crucial. Smaller values of $\tau$ may miss some critical points but also reduce the number of candidate triangles, thus diminishing the possibility of a false positive. Bigger values of $\tau$ ensure that no critical point is missed, but the chances of false positives increase. Figure~\ref{crit_a} shows an example where a critical point could be missed due to asymmetry.

The working hypothesis made in the second step of the algorithm can fail due to the existence of lines that are tangent at a critical point and very close to being tangent at another critical point of the graph. We call the latter \emph{quasi-multiple} tangent lines. 

Experimental results show that triangles formed with quasi-multiple tangent lines usually appear next to triangles that truly contain a critical point. In those cases, the same pair of critical lines with small slope appears between both triangles, resulting in closeby duplicate images of the same critical point.
Since lines close to horizontal are the ones really deciding where a critical point is, in the algorithm we impose that each pair of lines is used only once. If the same pair of lines passes inside another triangle, then the resulting approximate critical point will be very close to one that we already have, and thus there is no need to recompute it. Figure~\ref{crit_b} shows an example of a critical point detected by a triangle of quasi-multiple tangent lines. 

In most cases the algorithm retrieves all critical points correctly, as in Fig.\;\ref{crit_c}.
However, if we suppose that the data functions are unknown, there is no way to decide if a function is correctly approximated or not. For this reason, we normally restrict the algorithm to five tangent lines and do not iterate further. Instead, we have implemented a simple test that checks if local maxima and local minima alternate in the resulting graph.

We first applied the five-line algorithm to 100 functions generated as combinations of sines and cosines, and checked in which cases all the critical points were properly retrieved with at least 4 digits of precision. All functions shared the domain $[0, 1]$ and had between 20 and 30 critical points. 
This kind of functions may present periodicity, and critical points are usually symmetric. By using slopes of $30^\circ$ and $0.1^\circ$, and 
$\tau = 0.08$, we achieved $89\%$ of correctly approximated functions. 
Errors in the reconstruction tended to be due to false negatives, i.e., candidate critical points that had been left out in the first step of the algorithm, rather than false positives. 

A second test was performed over 100 functions generated by spline interpolation of third order over 30 points on the interval $[0, 1]$. Such functions are harder to approximate since inflection points, flat critical points and asymmetric critical points appear with more frequency than in the previous case. Again, we obtained around $90\%$ of properly approximated functions. 
For this kind of functions, errors regarding false positives and false negatives were balanced. 

\section{Approximating a function from selected landscape levels}

Landscapes of persistence diagrams were introduced in \cite{Bubenik2015} as follows. For each point $(\beta,\delta)$ in a given persistence diagram, one denotes
$
\Lambda_{(\beta, \delta)}(t) = \max \{0, \min \{t-\beta, \delta-t\}\}
$.
Then a piecewise linear function $\lambda_k \colon\mathbb{R}\to\mathbb{R}$ is defined for each $k\ge 1$ as
\[
\lambda_k(t) = {\rm kmax} \{ \Lambda_{(\beta_i, \delta_i)}(t)\},
\]
where $\{(\beta_i,\delta_i)\}$ is the set of all points in the given persistence diagram and kmax returns the $k$-th largest value of a given set of real numbers whose elements are possibly counted with multiplicities, or zero if there is no $k$-th largest value. Since the number of points in a persistence diagram is finite, there is an integer $n$ such that $\lambda_k=0$ for $k\ge n$.
The first landscape levels $\lambda_1,\lambda_2\dots$ depict the most significant features from the persistence diagram, while the last ones correspond to short-living phenomena or perhaps to noise.

\begin{figure}[htb]
    \centering
    
      \begin{subfigure}{\textwidth}
      \centering
          \includegraphics[width=0.24\textwidth]{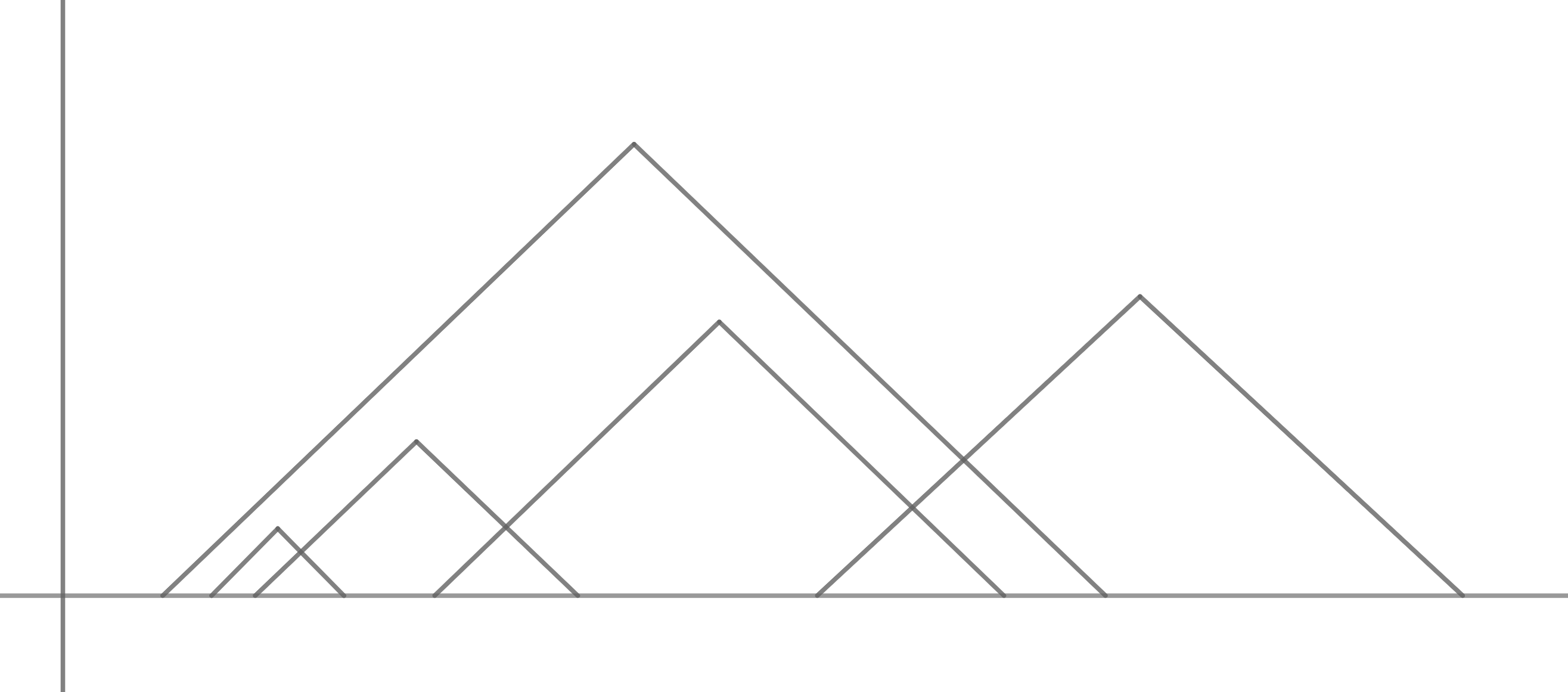} \hspace{0.0cm}
          \includegraphics[width=0.24\textwidth]{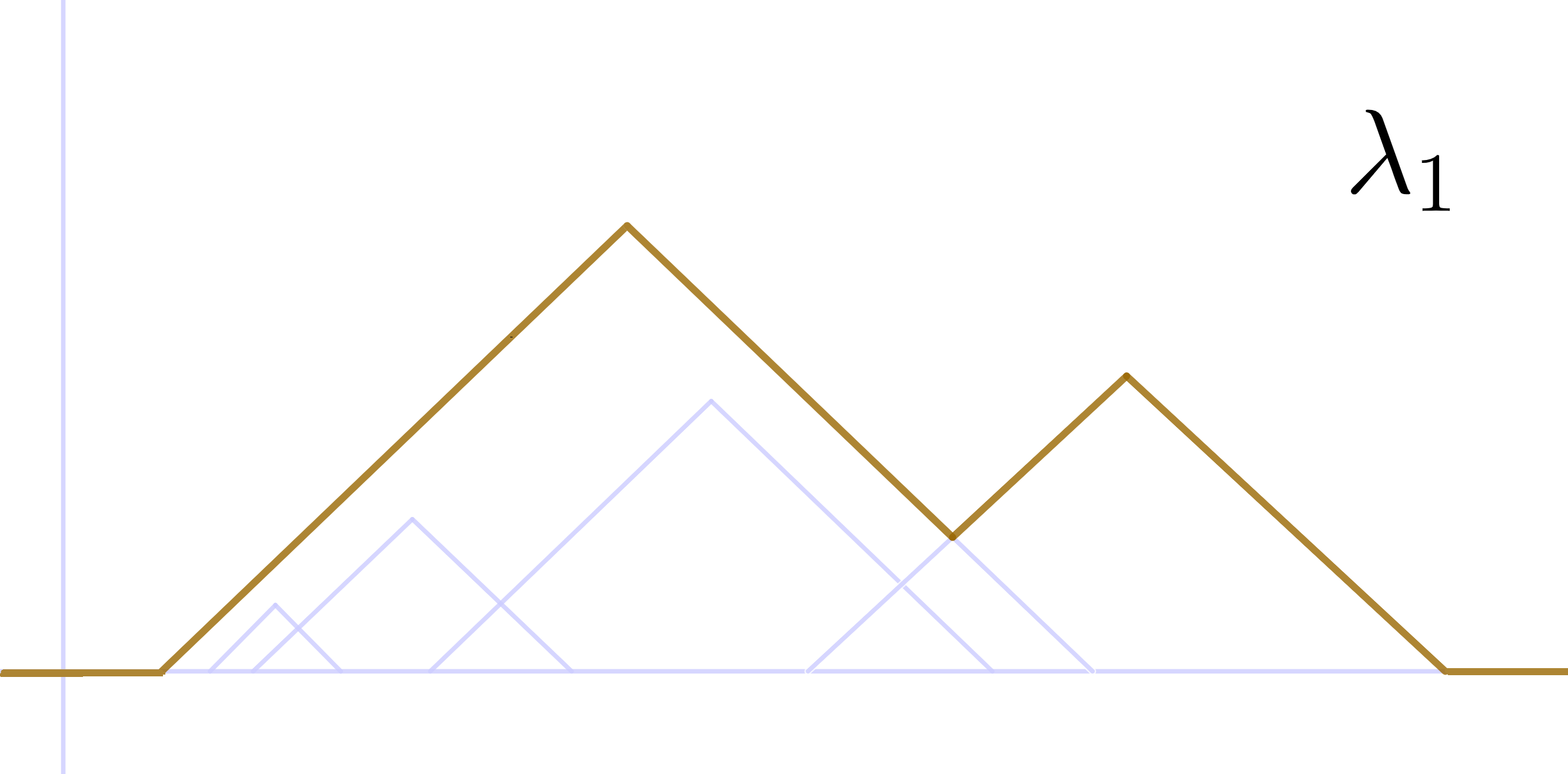} \hspace{0.0cm}
          \includegraphics[width=0.24\textwidth]{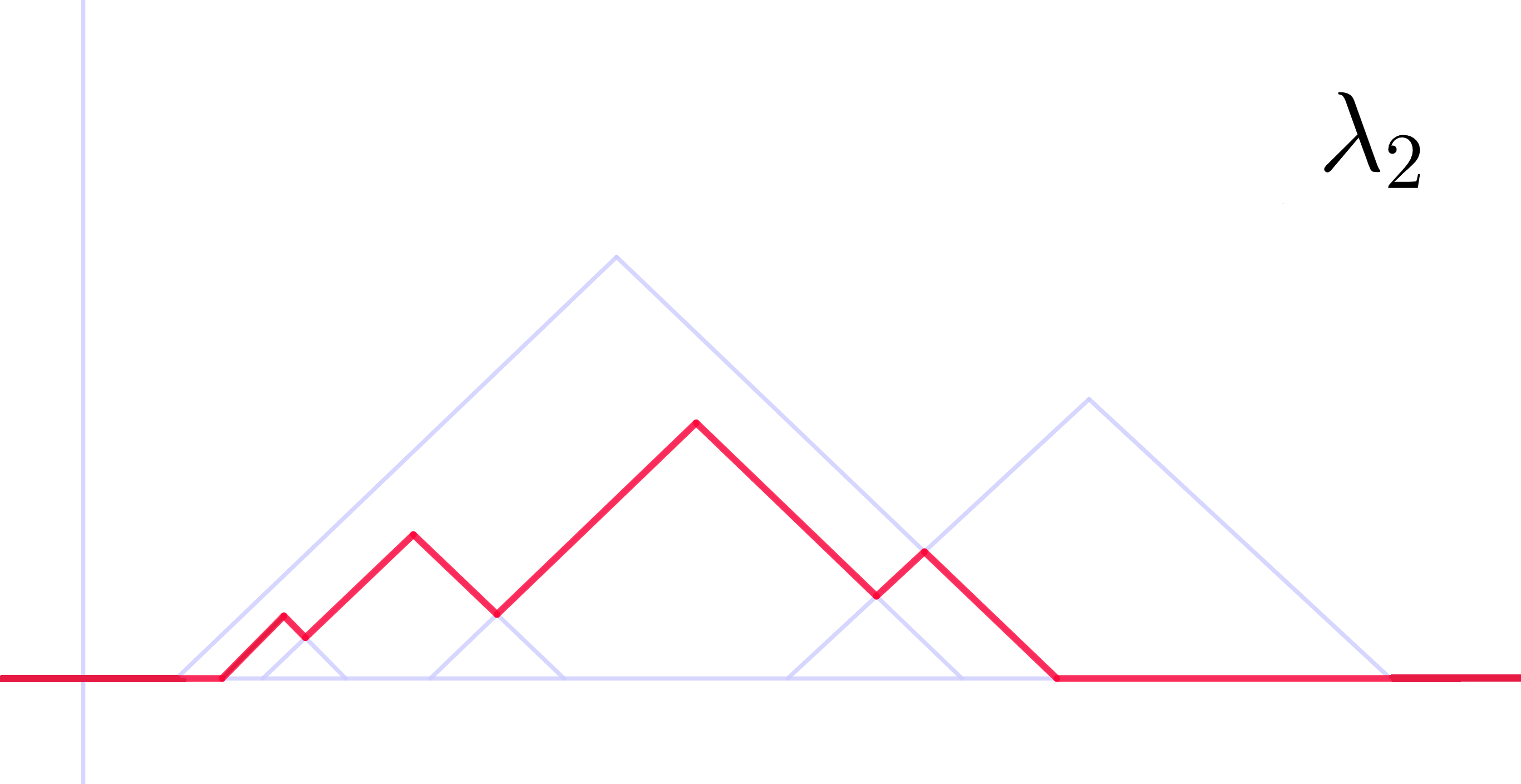} \hspace{0.0cm}
          \includegraphics[width=0.24\textwidth]{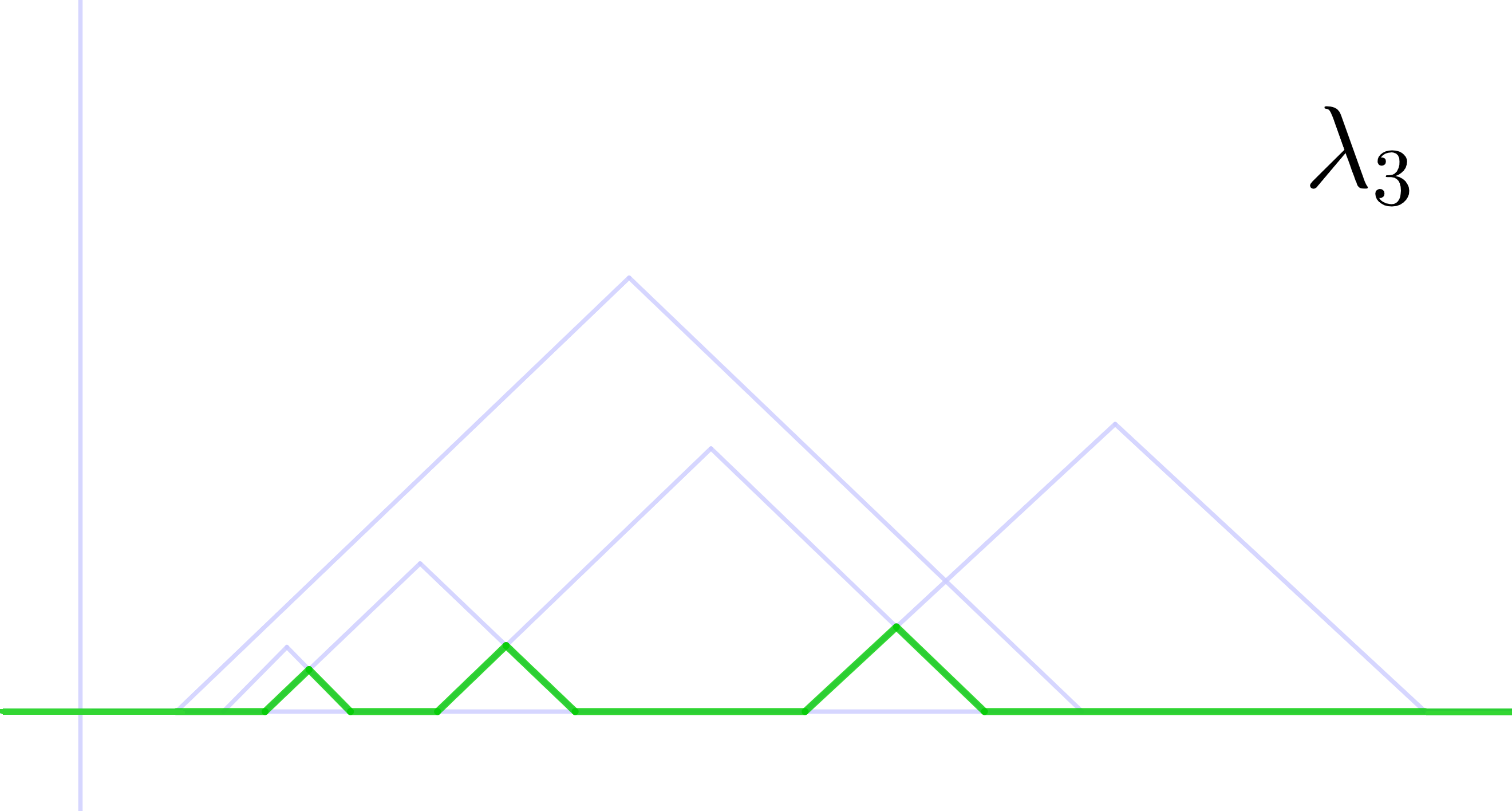}
        \end{subfigure}

    \caption{Sequence of nonzero levels of a persistence landscape (left).}
    \label{landscapes}
\end{figure}

Suppose given a continuous function $f\colon [a,b]\to\mathbb{R}$ and consider the persistence diagram of sublevel sets of $f$ in the vertical direction.
In this case, for each point $(\beta, \delta)$ in the persistence diagram, the coordinates $\beta$ and $\delta$ correspond to $y$-values of critical points of $f$ by Proposition~\ref{first}, except for $\delta=\infty$.
Hence it is possible to determine which critical points of $f$ are associated with each landscape level~$\lambda_k$. This is done under the assumption that the function $f$ is known, so our task is to find out which subset of its critical points correspond to~$\lambda_k$ for a given value of~$k$.

Our algorithm to do this is
the following.
Denote by $(t,u)$ the coordinates of landscape points. A point $(t,0)$ is called a \emph{take-off point} for a nonzero landscape function $\lambda_k$ if there is an $\varepsilon>0$ such that $\lambda_k(s)=0$ for $t-\varepsilon< s\le t$ and $\lambda_k(s)\ne 0$ for $t<s<t+\varepsilon$. Similarly, $(t,0)$ is a \emph{landing point} for $\lambda_k$ if there is an $\varepsilon>0$ such that $\lambda_k(s)\ne 0$ for $t-\varepsilon<s<t$ and $\lambda_k(s)= 0$ for $t\le s<t+\varepsilon$.

\begin{prop}
\label{flight}
Let $\lambda_k$ be a nonzero landscape level of a continuous function $f\colon [a,b]\to\mathbb{R}$. Let $\{t_i\}$, $i\ge 0$, be the ordered set of $t$-coordinates of vertices in the graph of $\lambda_k$, where $t_i<t_{i+1}$ for all~$i$. Then the sequence defined by $y_i=t_i$ if $(t_i,0)$ is a take-off point and $y_i=2\big(t_i-\frac12 y_{i-1}\big)$ otherwise is a sequence of $y$-values of critical points of~$f$.
\end{prop}

\begin{proof}
If $(t_i,0)$ is a take-off point of $\lambda_k$ (including the case $i=0$) then $t_i$ is the birth coordinate $\beta$ of a point $(\beta,\delta)$ in the persistence diagram of sublevel sets of~$f$, hence a $y$-value of some critical point of~$f$. 
If $(t_i,u_i)$ is a local maximum of $\lambda_k$, then $t_i=\frac12(\beta_p+\delta_q)$ and $u_i=\frac12(\delta_q-\beta_p)$ for points $(\beta_p,\delta_p)$ and $(\beta_q,\delta_q)$ in the persistence diagram (that may be equal or different), and $y_{i-1}=\beta_p$. Hence
\[
\textstyle
y_i=2\big(t_i-\frac12 y_{i-1}\big)=2\big(t_i-\frac12 \beta_p)=\delta_q,
\]
which is a $y$-value of a critical point of~$f$. 
If $(t_i,u_i)$ is a local minimum of $\lambda_k$, then $t_i=\frac12(\beta_p+\delta_q)$ and $u_i=\frac12(\delta_q-\beta_p)$ for some points $(\beta_p,\delta_q)$ and $(\beta_p,\delta_q)$ as well, yet now $y_{i-1}=\delta_q$. In this case,
\[
\textstyle
y_i=2\big(t_i-\frac12 y_{i-1}\big)=2\big(t_i-\frac12 \delta_q)=\beta_p.
\]
If $(\delta,0)$ is a landing point, then $y_{i-1}=\delta$ in the preceding maximum. Therefore $2\big(\delta-\frac12 y_{i-1}\big)=\delta$, which comes repeated and hence landing points can be omitted.
\end{proof}

This method produces a list of $y$-values of critical points of $f$ associated with a subset of selected landscape levels.
Next, the values in this list are compared with the list of all critical points of $f$ in order to obtain the matching $x$-values. 
The precise algorithm is detailed below as Algorithm~\ref{aquest}. If two or more critical points of $f$ coincidentally have the same $y$-value, the algorithm retrieves them simultaneously.

If all nonzero landscape levels $\lambda_k$ are used, then all the critical points of $f$ are recovered.
As shown in the proof of Proposition~\ref{flight}, local minima of $f$ are obtained from take-off points or local minima of landscape functions, while local maxima of $f$ are paired with local maxima of the corresponding landscape functions.
An example of a full reconstruction is shown in Fig.\;\ref{fig:rec_process_2}.
Since the reconstruction is done sequentially by levels,
if only a subset of landscape levels is provided then a partial reconstruction is obtained. 

\begin{figure}[htb]
    \centering
    \begin{subfigure}{\textwidth}
    \centering
    \includegraphics[width=0.8\textwidth]{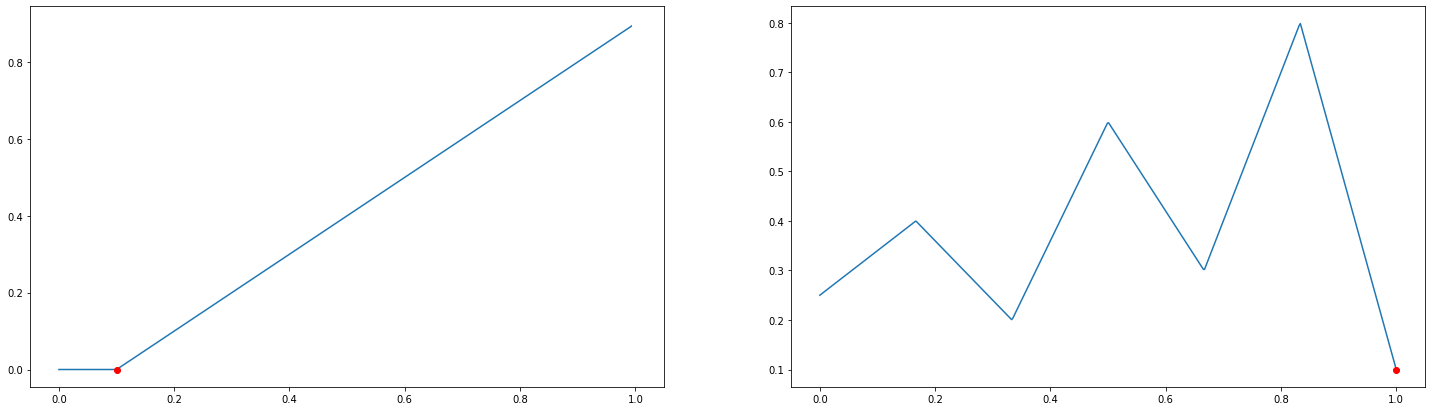}  
    \end{subfigure} \hspace{2em}
    \begin{subfigure}{\textwidth}
    \centering
    \includegraphics[width=0.8\textwidth]{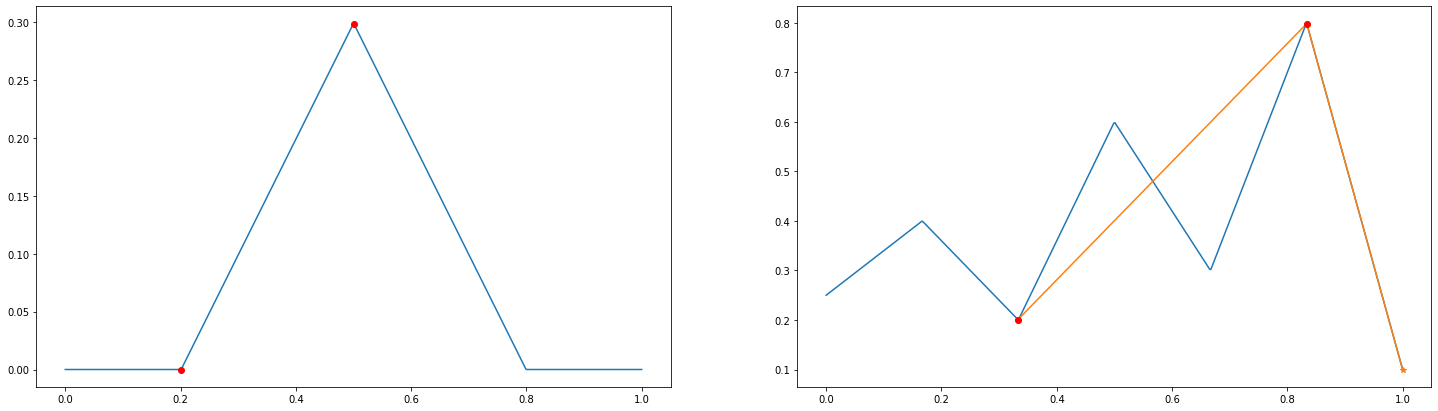}
    \end{subfigure}
    \begin{subfigure}{\textwidth}
    \centering
    \includegraphics[width=0.8\textwidth]{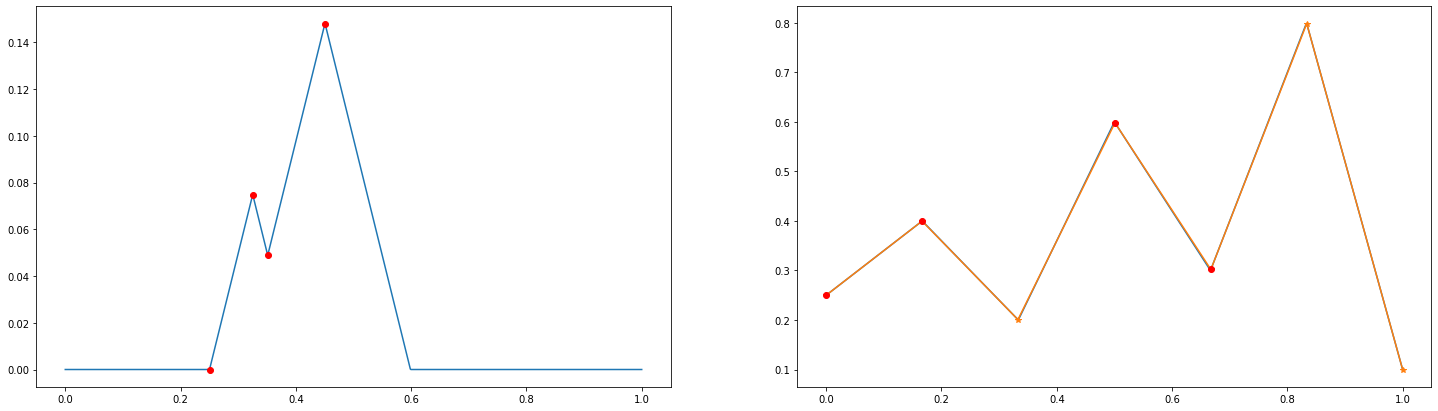}  
    \end{subfigure} \hspace{2em}
    \caption{Each landscape level (left) matches a subset of critical points of a given function (right). The pictures illustrate how the vertices of the successive landscape levels are paired with local extrema of the function (marked in red). Landscape heights have been rescaled.}
    \label{fig:rec_process_2}
\end{figure}

\subsection{An application to importance attribution in neural networks}
\label{section4}

Our work was initially motivated by a practical problem, namely finding out which information from signals is discriminative for classification by a neural network, which is an important study subject in machine learning.

Towards this goal, the aforementioned algorithm was applied to classification of a sample of 87,554 electrocardiogram signals (ECG) taken from the MIT-BIH Arrhythmia Database \cite{physionet,MoodyMark2001}, with the purpose of characterizing the relevant features that define each of five classes of interest: one corresponding to normal beats, three classes corresponding to different arrhythmia types, and one class for unidentifiable heartbeats. 
The database is accessible from~\cite{kaggle}.
A neural network consisting of three convolutional layers combined with max pooling layers followed by two dense layers was designed for the classification task. A custom gating layer was placed before the dense layer as described in~\cite{ACO2}.

Zero-homology of sublevel sets in the vertical direction was computed in order to pre-process each ECG sample and obtain its persistence landscapes. The resulting sequences of functions $\lambda_k$ were fed into a neural network for classification, and the custom gating layer was used to assign a weight to each landscape level in every sequence. Those landscape levels with the largest weights were considered as the ones conveying the most relevant information for the classification task. 
Next, Algorithm~\ref{aquest} was used to approximate the shape that contributed the most in characterizing each of the classes. 
This approximation only leverages information from the selected landscape levels. As a result, the data signals were partially reconstructed using only a subset of their persistence landscapes.

Table~\ref{accuracy} shows average accuracy (percentage of correctly classified samples) using a 5-fold cross-validation. The table compares the classification accuracy of a neural network fed with the original unprocessed signal with the same architecture using a set of ten levels of landscape decomposition, and the accuracy when using only the three most important landscape levels. The results show that the decomposition of the signals using persistence landscapes usually produces a small loss in classification accuracy. This was expected, as the decomposition process loses information in terms of complete signal reconstruction. Despite this fact, the results were close to each other ---average accuracy of 98.41\% with the original data functions versus 94.55\% for the set of the first 10 landscape levels and~94.00\% using only the most meaningful 3 landscape levels.

The simplified versions of data signals were subsequently fed into the same neural network
in order to find out if the data features emphasized by our reconstruction method were sufficient for the
network’s classifications task. The results can also be seen in Table~\ref{accuracy} and demonstrate that, indeed, the accuracy obtained with the simplified functions was comparable to that of the original data.

\begin{table}[htb]
\centering
\begin{tabular}{lcccc}
& \textbf{Raw data} & \;\textbf{10 levels} & \textbf{3 levels} & \textbf{Reconstructed} \\ \hline \\[-0.25cm]
\textbf{Accuracy} & $98.41 \pm 0.09$ & \;$94.55 \pm 0.16$     & $94.00 \pm 0.13$ & $97.04 \pm 0.14$ \\[0.1cm] \hline
\end{tabular}
\caption{Classification accuracy of a neural network fed with raw data versus processed data using ten landscape levels and the most significant three levels, as well as data reconstructed from three landscape levels.}
\label{accuracy}
\end{table}

The outcome of persistent homology adds considerable value towards a more comprehensive attribution along the classification problem. A ranking of landscape importance for the network can be seen in Fig.\;\ref{landscape_rank}. Observe that the first three landscape levels gather most of the attribution power. 

\begin{figure}[htb]
    \centering
    \includegraphics[width=0.5 \textwidth]{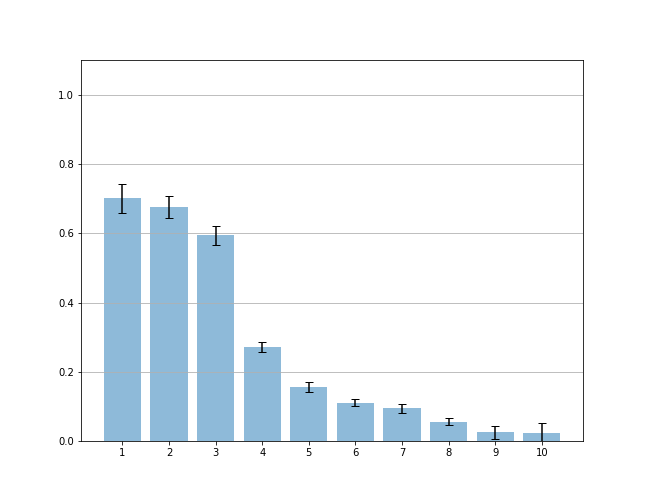}
    \caption{Rank of the importance of each landscape level based on results of a neural network. 
    Higher bars represent greater importance. Confidence intervals correspond to five repetitions of each classification.}
    \label{landscape_rank}
\end{figure}

\section{Discussion}

This article is a contribution to the general inverse problem of reconstructing data from persistence diagrams \cite{beltonetal,Fasyetal2019,Solomonetal2018}. 
While previous works focused mainly on simplicial complexes (in particular, on planar graphs), our method applies to both piecewise linear functions and smooth functions, albeit with distinct algorithms.
We implemented our techniques and tested it with examples from various sources, namely piecewise linear functions, smooth functions obtained using a harmonic function generator, and spline interpolations of third order.
We also used our method with data from a biomedical study \cite{MoodyMark2001} involving electrocardiographic signals, which had been classified into five classes. Our motivation was to address a case of the attribution problem in neural network classifiers.

Our algorithm in the piecewise linear case is a deterministic implementation, with only mild assumptions on the data functions, of the general procedure of using three directional persistence diagrams to reconstruct a planar graph. However, in the application to heartbeat registers, we  used a simplified form of the algorithm due to the fact that 
the neural network was fed with sequences of landscape levels, each of which derived from a sole persistence diagram.
Since the data functions remained available during the classification process, one persistence diagram was sufficient to find the relevant critical points among the set of all vertices of the given graphs. 
Our method allowed us to simplify the given graphs by substantially reducing the number of critical points while keeping essential information about what the neural network was focusing on. 
After the classification had been performed, we carried out a reduction of the number of critical points based on what the network regarded as important, thus highlighting the parts of the original function that were crucial. 

In the case of smooth functions,
five directional persistence diagrams are generally needed for reconstruction, except possibly in instances where the algorithm encounters a difficulty to locate some critical point, which can be solved by means of additional tangent lines. We stated a general result (Theorem~\ref{SmoothTheorem}) with assumptions guaranteeing the convergence of the algorithm, hence yielding an approximation of a given smooth function with a piecewise linear function with the same set of critical points, using persistence diagrams of sublevel sets.

In the smooth case, our algorithm can fail to detect some critical points, or instead converge to false positives if the assumptions of the reconstruction theorem are not fulfilled. In practice a trade-off is required when choosing the $\tau$ value in the algorithm in order to minimize the chances of errors of both kinds. 
Functions with a large number of critical points per length unit may require more careful choices of~$\tau$.

Also critical points where the second derivative vanishes or is very close to zero may cause detection difficulties. The nonvanishing assumption on $f''$ aims to prevent the existence of critical points that may be missed by the algorithm.
We excluded graphs with horizontal segments, although our algorithms could easily be modified in order to treat such cases as well. In the piecewise linear case, the~algorithm detects either the beginning or the end of each horizontal segment, and in order to detect the other end it would suffice to repeat the process using critical lines with reverse slope. In the smooth case, the algorithm still works in the case of sufficiently short horizontal segments (for which it finds the midpoint), although there is an increased risk of not detecting them. Indeed, detection of critical points with $f''$ close to $0$ requires larger values of~$\tau$.

False occurrences of critical points may be caused by the existence of multiple tangent lines or \emph{quasi-multiple} tangent lines, that is, lines that are tangent at a critical point and very close to being tangent at another critical point.
In order to prevent erroneous detection of nonexistent critical points, we perform a global checking, namely that critical points alternate between local maxima and local minima. This allows us to eliminate false positives in some~cases. 

There are two ways in which the numerical precision of the code can affect the result of our algorithm. On one hand, graphs are stored as vectors, and if the step is too large then flat segments may be created at critical points. On the other hand, if the threshold $\tau$ to detect a critical point is small compared with the code precision, then we could lose critical points.

If a given function $f$ is defined on an interval $[a,b]$, then the points $(a,f(a))$ and $(b,f(b))$ can only be detected if they are local minima. Otherwise the sublevel set does not change its topology when passing through them. To deal with this issue, the location of the initial point and the final point of the graph of $f$ should be given as input to the reconstruction algorithm.
A similar strategy could be used to deal with graphs starting and ending with horizontal segments, which is sometimes the case in practice.

While it was feasible to optimize our algorithm in the piecewise linear case, reducing the complexity in the smooth case would require a new design of the algorithm. In its current form, the detection step and the approximation step are successive. Detection of critical points is made in $O(n^3)$ time, where $n$ is the number of critical points, and further calculations are needed pointwise in order to determine each precise location. A more efficient algorithm could be produced by mixing the two steps into a simultaneous searching procedure.

This study illustrates another instance of the usefulness of persistence landscapes as descriptors from homological persistence diagrams. Thanks to the fact that landscapes are made of a hierarchical sequence of levels, it is possible to analyze the way in which information from data is distributed along the sequence of levels. As a consequence, landscapes are valuable as a tool for interpretability in classification processes.

Reconstruction of data is also related with privacy in data science~\cite{BDFKR2018}. In our case, knowledge of a persistence diagram of sublevel sets of a function does not permit to fully reconstruct the function, since the $x$-values of critical points cannot be recovered from a persistence diagram. However, as shown in this article, an oracle yielding enough directional persistence diagrams conveys as much information as the original functions. Towards applications, removing critical points that are not essential for classification purposes is a way to erase information stored in the data while preserving usability. Consequently, selective critical point removal could be a convenient method to diminish the chances of deanonymization when the data functions encode confidential content.

\printbibliography[heading=bibintoc]

{\begin{algorithm}[H]
\SetKwData{crits}{crits}

\SetKwInOut{Input}{Input}
\SetKwInOut{Output}{output}

\KwIn{$\theta_0, \theta_1$ angles, $t_0, t_1$ heights of the lines with angles $\theta_0, \theta_1$\;}
\KwOut{A point where the two given lines cross\;}
\BlankLine
\KwRet{$\left((t_1\sin\theta_0 -t_0\sin\theta_1)/\sin(\theta_0-\theta_1),\; (t_0\cos\theta_1-t_1\cos\theta_0)/\sin(\theta_0-\theta_1)\right)$}
\caption{Intersect function}
\end{algorithm}}

\vspace*{1cm}

{\begin{algorithm}[H]
\SetKwData{crits}{crits}
\SetKwData{T}{T}
\SetKwData{R}{R}
\SetKwData{S}{S}
\SetKwData{start}{start}
\SetKwData{fin}{end}
\SetKwData{True}{True}
\SetKwData{False}{False}
\SetKwInOut{Input}{Input}
\SetKwInOut{Output}{output}

\KwIn{\T, \S, \R persistence diagrams, $\theta_0$, $\theta_1$, $\theta_2$ directions, \start starting point and \fin last point of the function to be reconstructed\;}
\KwOut{List of all critical points of the given function\;}
\BlankLine
$\crits \leftarrow \crits \cup \{\start, \fin\}$\;
\For{$t \in T$}
    {\For{$r \in R$}
        {\For{$s \in S$}
        {
        $p_{tr} \leftarrow $\texttt{intersect}$(\theta_0, \theta_1, t, r)$\;
        $p_{ts} \leftarrow $\texttt{intersect}$(\theta_0, \theta_2, t, s)$\;
        \If{$|| p_{tr} - p_{ts} || < 10^{-6}$}{
            $\crits \leftarrow \crits \cup \{p_{tr}\}$\;
          }
        }
    }
}
\KwRet{\crits}
\caption{Naive reconstruction algorithm}
\label{algorithm:naive}
\end{algorithm}}

\vspace*{1cm}

{\begin{algorithm}[H]
\SetKwData{crits}{crits}
\SetKwData{T}{T}
\SetKwData{R}{R}
\SetKwData{S}{S}
\SetKwData{start}{start}
\SetKwData{fin}{end}
\SetKwData{True}{True}
\SetKwData{False}{False}
\SetKwInOut{Input}{Input}
\SetKwInOut{Output}{output}

\KwIn{\T, \S, \R persistence diagrams, $\theta_0$, $\theta_1$, $\theta_2$ directions, \start starting point and \fin last point of the function to be reconstructed\;}
\KwOut{List of all critical points of the given function\;}
\BlankLine
$\crits \leftarrow \crits \cup \{\start, \fin\}$\;
\texttt{sort}(\T, decreasing)\;
\texttt{sort}(\S, increasing)\;
\For{$s \in S$}
    {$i = 0, j = 0$\;
    foundTriple = \False\;
     \While{\text{!foundTriple} and $i < |R|$ and $j < |T|$}{
         $p_r \leftarrow $\texttt{intersect}$(\theta_1, \theta_2, s, R_i)$\;
         $p_t \leftarrow $\texttt{intersect}$(\theta_1, \theta_0, s, T_j)$\;
         \uIf{$|| p_r - p_t || < 10^{-6}$}{
            foundTriple = \True\;
            $\crits \leftarrow \crits \cup \{p_r\}$\;
          }
          \uElseIf{$x(p_r) < x(p_t)$}{
            $i = i+1$\;
          }
          \Else{
            $j = j+1$\;
          }
     }
}
\KwRet{\crits}
\caption{Rolling ball algorithm for piecewise linear functions}
\label{algorithm:optimized}
\end{algorithm}}

\newpage

{\begin{algorithm}[H]
\SetKwData{crits}{crits}

\SetKwInOut{Input}{Input}
\SetKwInOut{Output}{output}

\KwIn{$x_1$, $x_2$, $x_3$, $x_4$ $x$-coordinates near a critical point on a horizontal line,\newline $m_1$, $m_2$ slopes of critical lines\;}
\KwOut{An approximation to the $x$-coordinate of the critical point\;}
\BlankLine
\KwRet{$[m_2(x_1+x_2)(x_3-x_4)-m_1(x_1-x_2)(x_3+x_4)]/[2m_2(x_3-x_4)-2m_1(x_1-x_2)]$}
\caption{Compute the $x$-coordinate of a critical point}
\end{algorithm}}

\vspace*{1cm}

{\begin{algorithm}[H]
\SetKwData{crits}{crits}
\SetKwData{T}{T}
\SetKwData{R}{R}
\SetKwData{S}{S}
\SetKwData{R2}{R_2}
\SetKwData{S2}{S_2}
\SetKwData{triangles}{triangles}
\SetKwData{passtriangles}{pass\_triangles}
\SetKwData{used}{used}

\SetKwData{original}{original}

\SetKwInOut{Input}{Input}
\SetKwInOut{Output}{output}

\KwIn{$\T, \S, \R$ persistence diagrams corresponding to $\pi/2$, $\theta_1$, $-\theta_1$ respectively, \newline $\S2$, $\R2$ persistence diagrams corresponding to $\theta_2$,~$-\theta_2$ respectively, \newline $\tau$ detection parameter\;}
\KwOut{List of all critical points of the given function\;}
\BlankLine
$\crits \leftarrow \{ \}$\;
$\triangles \leftarrow  \{ \}$\;

\For{$t \in T$}
    {\For{$r \in R$}
        {\For{$s \in S$}
        {
        $p_{tr} \leftarrow $\texttt{intersect}$(\pi/2, -\theta_1, t, r)$\;
        $p_{ts} \leftarrow $\texttt{intersect}$(\pi/2, \theta_1, t, s)$\;
        \If{$| (p_{tr})_x - (p_{ts})_x | < \tau $}{
            $\triangles \leftarrow \triangles \cup \{(t, s, r)\}$\;
          }
        }
    }
}
$\passtriangles \leftarrow \{ \}$\;
$\used \leftarrow \{ \}$\;
\For{$(t, r, s) \in \triangles$}{
    \For{$ss \in \S2$}{
        \For{$rr \in \R2$}{
            \uIf{$(rr, ss)\text{ between }(t, r, s) \text{ and } (rr, ss) \not\in \used$}{
                $\passtriangles \leftarrow \passtriangles \cup \{(t, r, s, rr, ss)\}$\;
                $\used \leftarrow \used \cup \{(rr, ss)\}$\;
            }
        }
    }
}

\For{$(t, r, s, rr, ss) \in \passtriangles$}{
    $x_1 \leftarrow $\texttt{intersect}$(\pi/2, -\theta_1, t, r)$\;
    $x_2 \leftarrow $\texttt{intersect}$(\pi/2, \theta_1, t, s)$\;
    $x_3 \leftarrow $\texttt{intersect}$(\pi/2, -\theta_2, t, rr)$\;
    $x_4 \leftarrow $\texttt{intersect}$(\pi/2, \theta_2, t, ss)$\;
    $x \leftarrow $\texttt{compute\_x}$((x_1)_x, (x_2)_x, (x_3)_x, (x_4)_x, 1/\tan\theta_1, 1/\tan\theta_2)$\;
    $y \leftarrow (x_1)_y$\;
    $\crits \leftarrow \crits \cup \{(x, y)\}$\;
}
$\crits \leftarrow $\texttt{sort}$(\crits)$\;
\KwRet{$\crits$}
\caption{Five-line algorithm for smooth functions}
\end{algorithm}}

\newpage

{\begin{algorithm}[H]
\SetKwData{crits}{crits}
\SetKwInOut{Input}{Input}
\SetKwInOut{Output}{output}

\KwIn{$l$ vector of landscape data, $t$ vector of $t$-coordinates of landscape data\;}
\KwOut{List of $y$-values of critical points\;}
\BlankLine
$\crits \leftarrow \{ \}$\;
\For{$i\gets0$ \KwTo $|l|$}{
    \uIf{$l_i$ is a take-off vertex}{
         $\crits \leftarrow \crits \cup \{t_i\}$\;
    }\uElseIf{$l_i$ is a local minimum or a local maximum}{
        $\crits \leftarrow \crits \cup \{t_i\}$\;
        $\crits_i \leftarrow 2(\crits_i - 0.5\,\crits_{i-1})$\;}
    }
\KwRet{$\crits$}
\caption{Get $y$-values of critical points}
\end{algorithm}
}

\vspace*{1cm}

{\begin{algorithm}[H]
\SetKwData{crits}{crits}
\SetKwData{xcrits}{xcrits}
\SetKwData{candcrits}{candcrits}

\SetKwInOut{Input}{Input}
\SetKwInOut{Output}{output}

\KwIn{$\crits$ list of $y$-values of critical points, $original$ vector of original data, $x$ vector of $x$-coordinates of original data\;}
\KwOut{List of $x$-values of critical points\;}
\BlankLine
$\candcrits \leftarrow \{ \}$\;
\For{$crit \in \crits$}{
    $compare \leftarrow (original - crit)^2$\;
    \For{$i\gets0$ \KwTo $|compare|$}{
        \uIf{$compare_i < 10^{-4}$}{
        $\candcrits \leftarrow \candcrits \cup \{i\}$\;
        }
    }
}

$\xcrits \leftarrow \{ \}$\;
\For{$i \in \candcrits$}{
    \uIf{$original_i$ is a local minimum or a local maximum}{
        $\xcrits \leftarrow \xcrits \cup \{x_i\}$\;
        }
}
\KwRet{$\xcrits$}
\caption{Get $x$-values of critical points}
\end{algorithm}
}

\vspace*{1cm}

{\begin{algorithm}[H]
\SetKwData{crits}{crits}
\SetKwData{landscapes}{landscapes}
\SetKwData{original}{original}

\SetKwInOut{Input}{Input}
\SetKwInOut{Output}{output}

\KwIn{$\landscapes$ a list of landscape levels, $original$ vector of original data\;}
\KwOut{List of selected critical points of the given function\;}
\BlankLine
$\crits \leftarrow \{ \}$\;
\For{$land \in \landscapes$}{
    $ candcrits \leftarrow \{ $\texttt{get\_y\_values}$(land) \}$\;
    $\crits \leftarrow \crits \cup \{$\texttt{get\_x\_values}$(candcrits, original) \} $\;
    }
$\crits \leftarrow $\texttt{sort}$(\crits)$\;
\KwRet{$\crits$}
\caption{Subset of critical points of a function associated with a set of landscape levels}
\label{aquest}
\end{algorithm}
}

\end{document}